\documentclass[11pt]{article}

\usepackage{graphicx}
\usepackage{graphics}
\usepackage{amssymb}
\usepackage{amsmath}
\usepackage{amsthm}
\usepackage{psfrag}
\usepackage{enumerate}

\usepackage{epsfig}
\usepackage{epsf}
\usepackage{amsfonts}
\usepackage{latexsym}
\usepackage{fancyvrb}
\usepackage[small]{caption}
\usepackage{url}
\usepackage {float}
\usepackage{subfigure}

\newtheorem{theorem}{Theorem}
\newtheorem{lemma}{Lemma}
\newtheorem{corollary}{Corollary}

\newtheorem {claim}{Claim}

\newcommand{\conv}{{\rm conv}}

\newcommand{\len}{{\rm len}}
\newcommand{\degree}{{\rm deg}}

\newcommand{\etal}{{et~al.}}
\newcommand{\ie}{{i.e.}}
\newcommand{\eg}{{e.g.}}

\newcommand{\old}[1]{{}}
\newcommand{\later}[1]{{}}

\graphicspath{{figures/}}

\title{Counting Carambolas\footnote{A preliminary version of this paper appeared in the
Proceedings of the 25th Canadian Conference on Computational Geometry~\cite{LST13}.}}

\author
{Adrian Dumitrescu\thanks{Department of Computer Science,
University of Wisconsin--Milwaukee, USA\@.
Email: \texttt{dumitres@uwm.edu}}
\and
Maarten L\"offler\thanks{Department of Computing and Information Sciences,
Utrecht University,  Utrecht, the Netherlands. Email: \texttt{m.loffler@uu.nl}}
\and
Andr\'e Schulz\thanks{LG Theoretische Informatik,
FernUniversit\"at Hagen, Germany. Email: \texttt{andre.schulz@fernuni-hagen.de}}
\and
Csaba D. T\'oth\thanks{Department of Mathematics, California State
  University Northridge, Los Angeles, CA, USA; and
  Department of Computer Science, Tufts University, Medford, MA, USA.
  Email: \texttt{cdtoth@acm.org}}
}

\begin{document}

\maketitle

\begin {abstract}
We give upper and lower bounds on the maximum and minimum number of
geometric configurations of various kinds present (as subgraphs) in a
triangulation of $n$ points in the plane. Configurations of interest include
\emph{convex polygons}, \emph{star-shaped polygons} and \emph{monotone paths}.
We also consider related problems for \emph{directed} planar straight-line graphs.

\bigskip
\textbf{\small Keywords}: convex polygon, star-shaped polygon, monotone path,
plane graph, triangulation, counting.
\end {abstract}

\section {Introduction} \label{sec:intro}

We consider \emph{plane straight-line graphs} (also referred to as
\emph{planar geometric graphs}),
where the vertices are points in the plane and the edges are line
segments between the corresponding points,
no two of which intersect except at common endpoints.
According to a classical result of Ajtai~\etal~\cite{ACNS82}, the
number of plane straight-line graphs on $n$ points in
the plane is $O(c^n)$, where $c$ is a large absolute constant.

Problems in extremal graph theory typically ask for the minimum or maximum number
of certain subgraphs, \eg, perfect matchings, spanning trees or
spanning paths, contained in a graph of a given order.
Here we consider extremal problems in plane straight-line graphs,
where the classes of subgraphs are defined geometrically (\ie, membership
in the class depends on the coordinates of the vertices). For instance,
van~Kreveld, L\"offler, and Pach~\cite{KLP12} studied the number of convex polygons
(cycles) in geometric triangulations on $n$ points (vertices). They constructed
$n$-vertex triangulations containing $\Omega(1.5028^n)$ convex polygons, and
proved that every triangulation on $n$ points in the plane contains $O(1.6181^n)$
convex polygons. Dumitrescu and T\'oth~\cite{DT15} subsequently sharpened the upper bound
to $O(1.5029^n)$, thereby almost closing the gap between the upper and lower bounds.
\begin{figure}[htpb]
\centering
\includegraphics[scale=0.9]{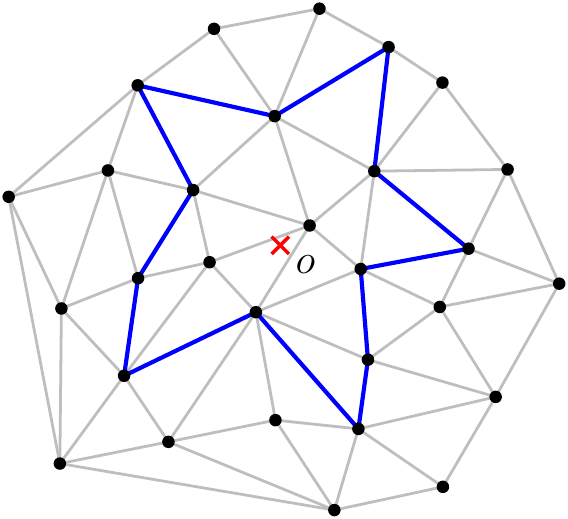}
\caption{A ``carambola'' in a triangulation.}
\label{fig:intro-example}
\end{figure}

In this paper we continue this research direction and investigate the
multiplicities of other geometrically defined subgraphs present in a
geometric triangulation, such as star-shaped polygons and monotone paths.
A \emph{star-shaped} polygon (a.k.a. \emph{carambola}, see
Figure~\ref{fig:intro-example}) is a simple polygon $P$ such that there
is a (center) point $o$ in its interior with the property that every
ray emanating from $o$ intersects the boundary of $P$ in exactly one
point. As we will see, star-shaped polygons are closely related to
monotone paths.

Let $\mathbf{u}\in \mathbb{R}^2\setminus \{\mathbf{0}\}$ be a nonzero vector
(here $\mathbf{0}=(0,0)$).
A polygonal path $\xi=(v_1,v_2,\ldots, v_t)$ is \emph{monotone in direction $\mathbf{u}$},
if every line orthogonal to $\mathbf{u}$ intersects $\xi$ in at most one point.
Equivalently, a path $\xi$ is \emph{monotone} in direction $\mathbf{u}$,
if every directed edge of $\xi$ has a positive scalar product with $\mathbf{u}$,
that is, $\langle \overrightarrow{v_iv_{i+1}},\mathbf{u}\rangle>0$ for $i=1,\ldots, t-1$.
A special case is an $x$-monotone path, which is monotone in the horizontal direction
$\mathbf{u} =(1,0)$.
A path $\xi=(v_1,v_2,\ldots, v_t)$ is \emph{monotone} if it is monotone in some direction
$\mathbf{u}\in \mathbb{R}^2\setminus \{\mathbf{0}\}$.
Monotone paths are traditionally used in optimization; a classic example is the simplex algorithm
for linear programming, which traces a monotone path on the 1-skeleton of a $d$-dimensional
polytope of feasible solutions.

Counting star-shaped polygons and monotone paths in a triangulation $T$ can be reduced
to counting directed cycles and paths, respectively, in some orientation of $T$.
Indeed, let $o$ be a center point and orient every edge $ab$ in $T$ as $(a,b)$ iff $\Delta{oab}$
is a clockwise triangle. Then the number of star-shaped polygons centered at $o$
equals the number of directed cycles. Similarly, let $\mathbf{u}\in
\mathbb{R}^2\setminus \{\mathbf{0}\}$ be a nonzero vector, and orient
every edge $ab$ in $T$ as $(a,b)$ iff $\langle \overrightarrow{ab},\mathbf{u}\rangle>0$.
Then the number of $\mathbf{u}$-monotone paths in $T$ equals the number of directed paths.
To see the role played by the geometric constraints in the definition of these special
orientations, we also derive bounds on the maximum and minimum number of directed paths
in an oriented $n$-vertex triangulation.

\paragraph{Our results.}
In Sections~\ref{sec:lower} and \ref{sec:upper}, we derive lower and upper bounds,
respectively, on the \emph{maximum} number of subgraphs of a certain
kind that a triangulation on $n$ points can contain.
Table~\ref{tab:results} summarizes known and new results.
\begin {table} [htbp]
\centering
\begin{tabular}{|lll|}
\hline
Configurations & Lower bound & Upper bound\\
\hline
\hline
Convex polygons  & $\Omega(1.5028^n)$~\cite{KLP12} & $O(1.5029^n)$~\cite{DT15}\\
Star-shaped polygons & $\Omega(1.7003^n)$ &  $O(n^3\alpha^n)$\\
Monotone paths & $\Omega(1.7003^n)$ &  $O(n\alpha^n)$\\
Directed  simple paths       & $\Omega(\alpha^n)$ &  $O(n^2 3^n)$\\
\hline
\end{tabular}
\caption {Bounds for the maximum number of configurations in an $n$-vertex plane graph.
Results in row~1 are included for comparison; the bounds in rows 2-4 are proved in the paper.
Row~4 concerns directed graphs.
Note: $\alpha =1.8392\ldots$ is the unique real root of the cubic equation $x^3-x^2-x-1=0$.
\label {tab:results}}
\end {table}

In Section~\ref{sec:min}, we study the \emph{minimum} number of configurations
that a triangulation on $n$ vertices can contain. Our asymptotic bounds are summarized in
Table~\ref{tab:results-min}; more precise estimates are available in the respective section.
\begin {table} [htbp]
\centering
\begin{tabular}{|lll|}
\hline
Configurations & Lower bound & Upper bound\\
\hline
\hline
Convex polygons  & $\Omega(n)$ & $O(n)$\\
Star-shaped polygons & $\Omega(n)$ &  $O(n^2)$\\
Monotone paths & $\Omega(n^2)$ &  $O(n^{3.17})$\\
Directed  paths       & $\Omega(n)$ &  $O(n)$\\
\hline
\end{tabular}
\caption {Bounds for the minimum number of configurations in a triangulation
with $n$ vertices. Row~4 concerns directed triangulations.
\label {tab:results-min}}
\end {table}

\paragraph{Related work.}
Previous research studied the maximum number of cycles and spanning
trees in triangulations with $n$ vertices. Since cycles and spanning
trees, in general, have no geometric attributes, upper bounds on these
numbers hold for all edge-maximal planar graph, \ie, combinatorial triangulations.
Buchin~\etal~\cite{BKK07} showed that every triangulation with $n$ vertices
contains $O(2.8928^n)$  simple cycles, and there are triangulations that contain
$\Omega(2.4262^n)$ simple cycles and $\Omega(2.0845^n)$ Hamiltonian cycles.
Buchin and Schulz~\cite{BS10} proved that every $n$-vertex
triangulation contains $O(5.2852^n)$ spanning trees. These techniques
are instrumental for bounding the total number of noncrossing Hamiltonian
cycles and spanning trees that $n$ points in the plane admit~\cite{HSS13,SSW13}.
Some recent lower bounds on these numbers appear in~\cite{DSST13}; see
also~\cite{AHV+06,DT12,HSS13} and the references therein for other related results.

\section {Lower bounds on the maximum number of subgraphs} \label{sec:lower}

\subsection{Monotone paths}

We construct plane straight-line graphs on $n$ vertices that contain
$\Omega(1.7003^n)$ $x$-monotone paths; see Figure~\ref{fig:construction-monotone}.
Thus by orienting all edges from left to right, we obtain a directed plane straight-line graph
that contains $\Omega(1.7003^n)$ directed paths.
The same construction yields lower bounds for a few related subgraphs.
By arranging three copies of this graph around the origin in a cyclic fashion
as shown in Figure~\ref{fig:construction-star}, we obtain a plane straight-line graph
that contains $\Omega(1.7003^n)$ star-shaped polygons.
By connecting the leftmost and rightmost vertices by two extra edges
to a common vertex, we obtain a plane straight-line graph that contains
$\Omega(1.7003^n)$ monotone polygons, (defined as usual; see~\eg,~\cite[p.~49]{BCKO08}).
\begin{figure}[htpb]
\centering
\includegraphics[scale=1.2]{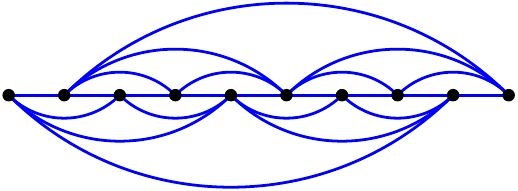}
\hspace{.5in}
\includegraphics[scale=1.2]{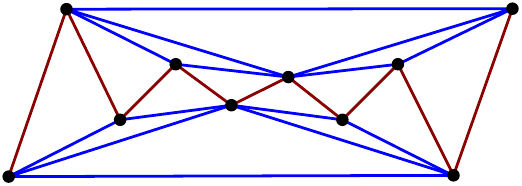}
\caption{Left: a graph on $n=2^\ell+2$ vertices (here $\ell=3$), that
  contains $\Omega(1.7003^n)$ monotone paths, for $n$ sufficiently large.
Right: a~straight-line embedding of the same graph where the points
lie alternately on two circular arcs, while preserving $x$-monotonicity.}
\label{fig:construction-monotone}
\end{figure}

Let $n=2^\ell+2$ for an integer $\ell\in \mathbb{N}$. We define a
plane graph $G$ on $n$ vertices $V=\{v_1,\ldots , v_n\}$: it consists
of a path $\pi = (v_1,\ldots , v_n)$ and two balanced binary triangulations
of the vertices $\{v_1,\ldots, v_{n-1}\}$ and $\{v_2,\ldots , v_n\}$,
respectively, one on each side of the path; see Figure~\ref{fig:construction-monotone}\,(left).
Specifically, $G$ contains an edge $(v_i,v_{i+2^k})$, for $1\leq i \leq n-2^k$, if and only if
$i-1$ or $i-2$ is a multiple of $2^k$. A straight-line embedding is shown in
Figure~\ref{fig:construction-monotone}\,(right), where the odd and respectively
the even vertices lie on two convex polygonal chains with opposite orientations.

\begin{theorem}\label{thm:lb}
  The graph $G$ described in the preceding paragraph has
  $\Omega(1.7003^n)$ $x$-monotone paths.
\end{theorem}
\begin{proof}
We count the number of $x$-monotone paths in a sequence of subgraphs of $G$. Let $G_0$
be the path $\pi= (v_1,\ldots , v_n)$; and we recursively define $G_k$ from $G_{k-1}$
by adding the edges $(v_i,v_{i+2^k})$ for $i=j2^k+1$ and $i=j2^k+2$
for $j=0,1,\ldots, 2^{\ell-k}-1$. The final graph is $G_\ell$.
Denote by $p_k(v_i)$ the number of $x$-monotone paths in $G_k$ that
end at vertex $v_i$. Since every monotone path can be extended to the
rightmost vertex $v_n$, the number of maximal (with respect to containment)
monotone paths in $G_k$ is $p_k(v_n)$.

We establish the following recurrence relations for $p_k(v_i)$. The
initial values are $p_k(v_1)=p_k(v_2)=1$ and $p_k(v_3)=2$ for all $k=0,\ldots,\ell$.
For $k=1$ and $i\geq 3$, we have $p_1(v_i)=p_1(v_{i-1})+p_1(v_{i-2})$,
therefore $p_1(v_i)=F_i$, where $F_i$ is the $i$th Fibonacci number.
It is well known that $F_i= \Theta(\phi^n)$, where $\phi=(1+\sqrt5)/2=1.6180\ldots$,
so $p_1(v_n)=\Theta(\phi^n)$.

The recurrence for $p_k(v_i)$, $k\geq 2$, is more nuanced, due to the asymmetry
between the triangulations on the two sides of the path~$\pi$.
We partition the edges of graph $G_k$ into groups, each induced by $2^{k-1}+2$
consecutive vertices (with respect to $\pi$), such that every two
consecutive groups share two vertices.
Let $a_i$ denote the first vertex of group~$i$, and let $b_i$ be the second vertex of group~$i$.
Let the edge $(a_i,b_i)$ belong to group $i$ but not to group $i-1$.
We count the number of ways one can route an $x$-monotone path through a group.
A path through group~$i$ starts at either $a_i$ or $b_i$, and ends at either
$a_{i+1}$ or $b_{i+1}$. Thus, it is enough to keep track of four different types of paths.
We record the number of paths from $a_i$ or $b_i$ to $a_{i+1}$ or $b_{i+1}$
in a $2\times 2$ matrix $M_k$, such that
\begin{equation*}
M_k\cdot (p_k(a_i),p_k(b_i))^T=(p_k(a_{i+1}),p_k(b_{i+1}))^T.
\end{equation*}

\begin{figure}[htpb]
\centering
\includegraphics[width=0.95\textwidth]{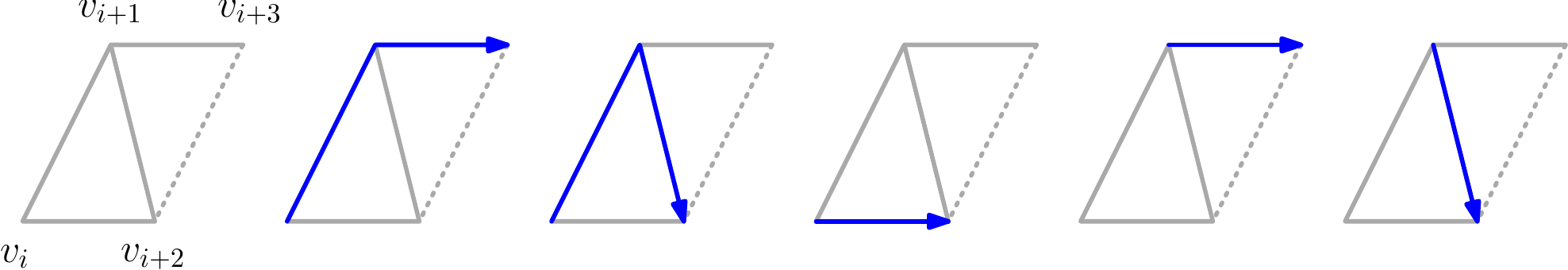}
\caption{The five possible $x$-monotone paths in the group of $G_2$.}
\label{fig:M2}
\end{figure}

Once the matrix $M_k$ is known, we can compute the number of paths by
$(p(v_{n-1}),p(v_n))^T=M_k^{(n-2)/2^{k-1}} \cdot (1,1)^T$.
By the Perron--Frobenius Theorem~\cite{HJ85},
$\lim_{q\to \infty} M_k^q/\lambda^q =  A$, for some matrix $A$, and
for $\lambda$ being the largest eigenvalue of $M_k$. Hence, we have
 $\lim_{n\to \infty}p_k(v_n)=\Theta(\lambda^{n/2^{k-1}})$ maximal $x$-monotone paths in $G_k$.

\begin{figure}[htpb]
\centering
\includegraphics[width=0.7\textwidth]{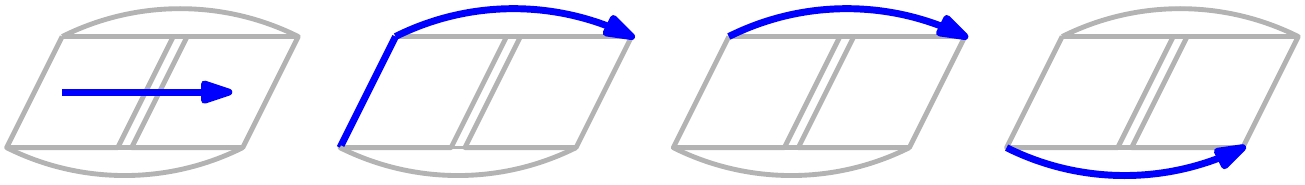}
\caption{Schematic drawing of the paths counted by $M_k$.}
\label{fig:Mk}
\end{figure}

We now show how to compute the matrices $M_k$ by induction on $k$.
The matrix $M_2$ can be easily obtained by hand (see Figure~\ref{fig:M2}).
For computing $M_k$, $k\geq 3$, consider an arbitrary group $i$ of size
$2^{k-1}+2$ in $G_k$. This group is composed of two consecutive groups
of $G_{k-1}$ that share two common vertices, say $a_j$ and $b_j$,
and two additional edges $a_ia_{i+1}$ and $b_ib_{i+1}$.
We distinguish two types of paths in group~$i$ of $G_k$:
(1) Paths that use only the edges in $G_{k-1}$.
Every such path is the concatenation of two paths,
from two consecutive groups of $G_{k-1}$,
with a common endpoint $a_j$ or $b_j$.
The number of these paths from
$a_i$ or $b_i$ to $a_{i+1}$ or $b_{i+1}$ is precisely $M_{k-1}^2$.
(2) The paths that use $a_ia_{i+1}$ or $b_ib_{i+1}$.
Since edge $a_ib_i$ is part of group~$i$, but edge $a_{i+1}b_{i+1}$ is not,
the only possible paths are $(a_i,a_{i+1})$, $(a_i,b_i,b_{i+1})$, and $(b_i,b_{i+1})$;
see Figure~\ref{fig:Mk}.
Therefore, we can compute the matrices $M_k$ iteratively as follows:
\begin{equation*}
 M_2:=\begin{pmatrix} 2 & 1 \\ 1 & 1 \end{pmatrix}, \text{ and }
 M_k:=M_{k-1}^2 +\begin{pmatrix} 1 & 0 \\ 1 & 1 \end{pmatrix}.
\end{equation*}

\begin{table}[htbp]
  \centering
  \begin{tabular}{@{} |c|c|c|c|c|c| @{}}

\hline
    $k$ & 2 & 3 & 4 & 5 & 6  \\
\hline
    $\lambda^{1/2^{k-1}}$ & 1.61803 & 1.69605 & 1.70034 & 1.70037 & 1.70037 \\
\hline

  \end{tabular}
  \caption{The asymptotic growth of the number of $x$-monotone paths in the graphs $G_k$.
    Already for $k=4$ there are $\Omega(1.7003^n)$ monotone paths.}
  \label{tab:eigenvalues}
\end{table}

Table~\ref{tab:eigenvalues} shows the values $\lambda^{1/2^{k-1}}$ for $k=2,\ldots,6$.
Note that when going from $k=5$ to $k=6$, there is no change in $\lambda^{1/2^{k-1}}$
up to 8 digits after the decimal point. The precise value of $\lambda$ for $k=5$
equals $\lambda=\frac12 (4885 + 9 \sqrt{294153})$.
\end{proof}

\subsection{Star-shaped polygons}

Use a projective transformation of the plane straight-line graph in
Figure~\ref{fig:construction-monotone} in which the point at infinity in direction $(0,1)$
moves to the center $o$ of the equilateral triangle in Figure~\ref{fig:construction-star}.
Use three copies of the resulting graph; the monotone order becomes a cyclic order
with respect to the center~$o$. The plane straight-line graph in Figure~\ref{fig:construction-star}
has $\Omega(1.7003^n)$ star shaped polygons, obtained by concatenating the images
of any three maximal monotone paths, one in each copy.
\begin{figure}[htpb]
\centering
\includegraphics[scale=1]{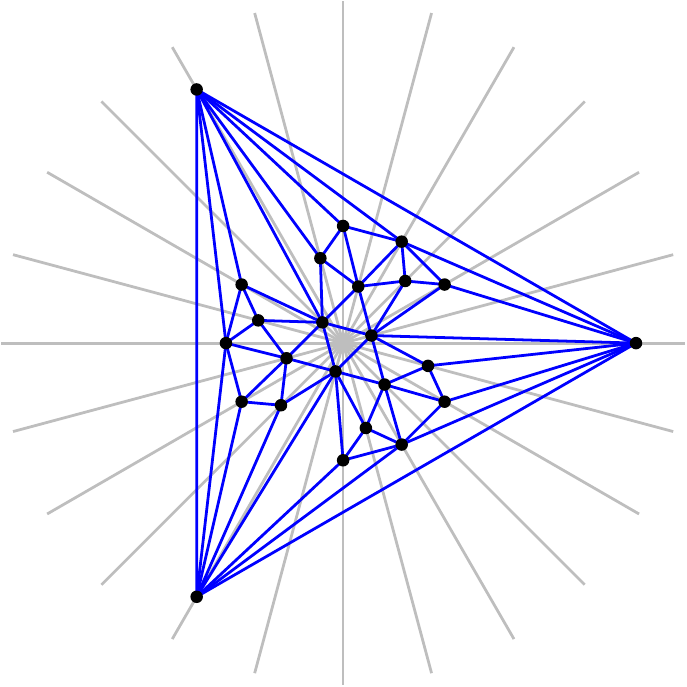}
\caption{Cyclic embedding of $3$ copies of the graph in Figure~\ref{fig:construction-monotone}.
The monotone order becomes a cyclic order.}
\label{fig:construction-star}
\end{figure}

\subsection{Directed plane graphs}

We construct $n$-vertex directed plane straight-line graphs that contain
$\Omega(1.8392^n)$ directed paths; Figure~\ref{fig:construction-paths}
shows the directed graph and Figure~\ref{fig:construction-paths-planar}
shows a planar embedding.
It is worth noting that the directed paths, however, cannot be
extended to a cycle because in a planar embedding
the start and end vertex are not in the same face. Similarly, most of the directed paths
are not monotone in any direction in a planar embedding.
\begin{figure}[htpb]
\centering
\includegraphics[scale=1.3]{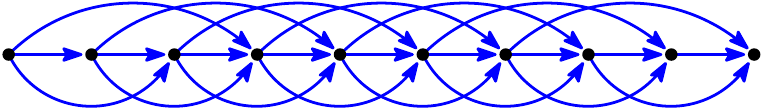}
\caption{There are $\Theta(\alpha^n)$ directed paths in this graph.
A plane embedding of the graph is depicted in Figure~\ref{fig:construction-paths-planar}.
Note: $\alpha =1.8392\ldots$ is the unique real root of the cubic equation $x^3-x^2-x-1=0$.}
\label{fig:construction-paths}
\end{figure}

\begin{figure}[htpb]
\centering
\includegraphics[scale=1.25]{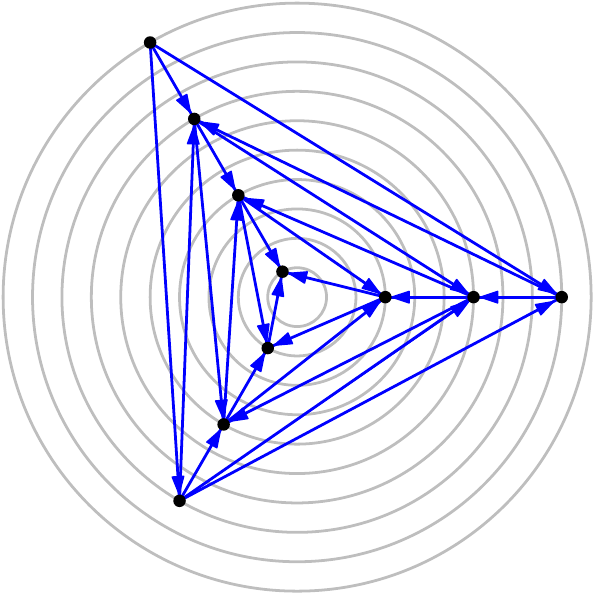}
\caption{A plane embedding of the graph in Figure~\ref{fig:construction-paths}.
The edges are directed from outer circles to inner circles.}
\label{fig:construction-paths-planar}
\end{figure}

Denoting by $T(i)$ the number of directed paths ending at vertex
$v_i$, we have $T(1)=T(2)=1$, $T(3)=2$, and a linear recurrence
relation
\begin{equation*}
T(i) = T(i-1) + T(i-2) + T(i-3), \text{ ~~for } i\geq 4.
\end{equation*}
The recurrence solves to $T(i)=\Theta(\alpha^i)$, where
$\alpha =1.8392\ldots$ is the unique real root of the cubic equation $x^3-x^2-x-1=0$.
Therefore the total number of directed paths, starting at any vertex, is $\Theta(\alpha^n)$.

\section {Upper bounds on the maximum number of subgraphs} \label{sec:upper}

\subsection{Monotone paths}

We start with $x$-monotone paths in a plane straight-line graph.
We prove the upper bound for a broader class of graphs, since some of the operations in
our argument may not preserve straight-line edges.
A \emph{plane monotone graph} is a graph embedded in the plane
such that every edge is an $x$-monotone Jordan arc.

Let $n\in \mathbb{N}$, $n\geq 3$, and let $G = (V, E)$ be a plane
monotone graph with $|V|=n$ vertices that maximizes the number of
$x$-monotone paths. We may assume that
(i) the vertices have distinct $x$-coordinates (otherwise we can perturb the vertices without
decreasing the number of $x$-monotone paths) and
(ii) the vertices lie on the $x$-axis
(by applying a homeomorphism that affects the $y$-coordinates).
We may also assume that $G$ is fully triangulated (\ie, it is an edge-maximal planar graph),
since adding $x$-monotone edges can only increase the number of $x$-monotone
paths~\cite{PT04}.

Label the vertices in $V$ as $v_1, v_2, \ldots, v_n$, sorted by their $x$-coordinates.
Orient each edge $\{v_i,v_j\} \in E$ from left to right, \ie, from $v_i$ to $v_j$ if $i < j$.
Define the \emph{length} of an edge $\overline{v_iv_j}\in E$ by $\len(\overline{v_iv_j})=|i-j|=j-i$.

Consider an edge $\overline{v_iv_j} \in E$ that is not on the boundary of
the outer face. There are two bounded faces incident to $\overline{v_iv_j}$,
and at least two other vertices $v_k$ and $v_l$ that are adjacent to both $v_i$ and $v_j$.
Suppose that $i< k,l<j$, and without loss of generality, that $k<l$
(\ie, $i< k<l<j$). The \emph{flip operation} for the edge $\overline{v_iv_j}$ replaces
$\overline{v_iv_j}$ by the edge $\overline{v_kv_l}$; note that this operation preserves planarity
but may introduce curved edges.

\begin{lemma}\label{lem:flip}
If $i < k<l< j$ as described above, then flipping $\overline{v_iv_j}$ to
  $\overline{v_kv_l}$ produces a plane monotone graph $G'$ with at least
  as many $x$-monotone paths as $G$ (see Figure~\ref {fig:flipquad-in+flipquad-out}.)
\end{lemma}
\begin{figure}[htpb]
\centering
  \subfigure
  {
    \includegraphics{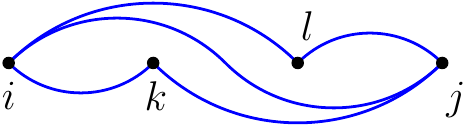}
    \label {fig:flipquad-in}
  }
   \subfigure
  {
    \includegraphics{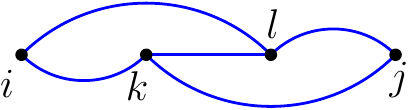}
    \label {fig:flipquad-out}
  }
\caption{The \textsf{flip} operation.}
\label{fig:flipquad-in+flipquad-out}
\end{figure}

\begin{proof}
The deletion of edge $\overline{v_iv_j}$ from $G$ creates a quadrilateral face
$(v_i,v_k,v_j,v_l)$. Since $i < k <l < j$, the new edge $\overline{v_kv_l}$
can be embedded in the interior of this face as an $x$-monotone Jordan arc.

Every $x$-monotone path in $G$ that does not contain $\overline{v_iv_j}$
is present in $G'$.
Define an injective map from the set of $x$-monotone paths that
traverse $\overline{v_iv_j}$ in $G$ into the set of $x$-monotone paths
that traverse $\overline{v_kv_l}$ in $G'$. To every $x$-monotone path $\xi$ that
traverses $\overline{v_iv_j}$ in $G$, map an $x$-monotone path $\xi'$
obtained by replacing edge $\overline{v_iv_j}$ with the path
$(v_i,v_k,v_l,v_j)$ in $G'$. It follows that $G'$ contains at least as
many $x$-monotone paths as~$G$.
\end{proof}

Note that the flip operation described in Lemma~\ref{lem:flip} decreases the
total length of the edges $\sum_{e\in E} \len(e)$. We may now
assume that among all $n$-vertex plane monotone graphs with the maximum
number of $x$-monotone paths, $G$ has minimal total edge length.
Thus Lemma~\ref{lem:flip} is inapplicable and we have the following.

\begin{lemma}\label{lem:flip2}
For every interior edge $\overline{v_iv_j}\in E$, with $i<j$,
there is a triangular face $(v_i,v_j,v_k)$ such that either $k<i<j$ or $i<j<k$.
\end{lemma}

We show next that $G$ contains an $x$-monotone Hamiltonian path.

\begin{lemma} \label {lem:path}
All edges $\overline{v_iv_{i+1}}$ are present in $G$.
\end{lemma}
\begin {proof}
Suppose, to the contrary, that there are two nonadjacent vertices
$v_i$ and $v_{i+1}$. Since $G$ is a  triangulation,  $v_iv_{i+1}$ is
not a boundary segment, and so there exists an edge that crosses the line
segment $v_iv_{i+1}$. Let $\overline{v_jv_k}$, $j<k$, be a longest
edge with this property.
Since the edge $\overline{v_jv_k}$ is $x$-monotone, we have $j < i < i+1 < k$.
The edge $\overline{v_jv_k}$ is not adjacent to the outer face, since
it crosses the segment $v_iv_{i+1}$ between two vertices.
Since edge $\overline{v_jv_k}$ is interior, by Lemma~\ref{lem:flip2},
there is a triangular face $(v_j,v_k,v_l)$ such that either $l<j<k$ or $j<k<l$.
Without loss of generality, assume that $j<k<l$.
Since there is no vertex in the interior of the face
$(v_j,v_k,v_l)$, the boundary of the face has to cross the segment
$v_iv_{i+1}$ twice: that is, $\overline{v_jv_l}$ crosses the segment $v_iv_{i+1}$.
Since $j<k<l$, we have $l-j>k-j$, thus $\overline{v_jv_l}$ is longer
than $\overline{v_jv_k}$, in contradiction to the assumption that $\overline{v_jv_k}$
is a longest edge crossing $v_iv_{i+1}$.
\end {proof}

For every pair $i < j$, let $V_{ij}$ denote the set of consecutive
vertices ${v_i, v_{i+1}, \ldots, v_j}$, and let $G_{ij} = (V_{ij},
E_{ij})$ be the subgraph of $G$ induced by $V_{ij}$.
Since $G$ is planar, we know that $|E| \leq 3|V| - 6$, and
furthermore, that $|E_{ij}| \leq 3|V_{ij}| - 6$ for all subgraphs
induced by groups of 3 or more consecutive vertices.

In the remainder of the proof we will apply a sequence of \textsf{shift} operations
on $G$ (defined subsequently) that may create multiple edges and edge crossings.
Hence, we consider $G$ as an abstract multigraph. However, the operations will
maintain the invariant that $|E_{ij}| \leq 3|V_{ij}| - 6$ whenever $|V_{ij}|\geq 3$.

Let $i < j < k$ be a triple of indices such that $\overline{v_iv_j}, \overline{v_iv_k}\in E$.
The operation \textsf{shift}$(i,j,k)$ removes the edge
$\overline{v_iv_k}$ from $E$, and inserts the edge $\overline{v_jv_k}$
into $E$ (see Figure~\ref{fig:shift-in+shift-out}). Note that the new
edge may already have been present, in this case we insert a new copy
of this edge (\ie, we increment its multiplicity by one).

\begin{figure}[htpb]
\centering
  \subfigure
  {
    \includegraphics{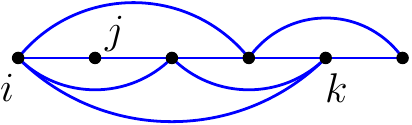}
    \label {fig:shift-in}
  }
   \subfigure
  {
    \includegraphics{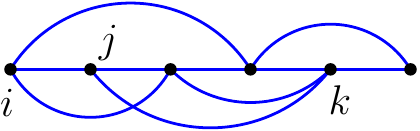}
    \label {fig:shift-out}
  }
\caption{The operation \textsf{shift}$(i,j,k)$.}
\label{fig:shift-in+shift-out}
\end{figure}

\begin {lemma}
The operation \textsf{shift}$(i,j,k)$ does not decrease the number of
$x$-monotone paths in $G$.
\end {lemma}
\begin {proof}
Clearly, any path that used $\overline{v_iv_k}$ can be replaced by a
path that uses $\overline{v_iv_j}$ and (the new copy of)
$\overline{v_jv_k}$.
\end {proof}

Now, we apply the following algorithm to the input graph $G$. We
process the vertices from left to right, and whenever we encounter a
vertex $v_i$ with outdegree 4 or higher, we identify the smallest
index $j$ such that $v_i$ has an edge to $v_j$ and the largest index
$k$ such that $v_i$ has an edge to $v_k$; and then apply
\textsf{shift}$(i,j,k)$. We repeat until there are no more vertices
with outdegree larger than $3$.

\begin {lemma}
The algorithm terminates and maintains the following two invariants:
\begin{enumerate} \itemsep -1pt
\item [\textup{(I1)}] All edges $\overline{v_iv_{i+1}}$ are present in $G$ with
  multiplicity one.
\item [\textup{(I2)}] $|E_{ij}| \leq 3|V_{ij}| - 6$ for all subgraphs induced
by $V_{ij}$, $i<j$.
\end{enumerate}
\end {lemma}
\begin {proof}
Initially, invariant (I1) holds by Lemma~\ref {lem:path}, and (I2) by
planarity. Edges between consecutive vertices are neither removed nor added
in the course of the algorithm, consequently (I1) is maintained.
To show that (I2) is maintained, suppose the contrary,
that there is an operation that increases the
number of edges of an induced subgraph $G'$
above the threshold. Let \textsf{shift}$(i,j,k)$ be the first
such operation. Since the only new edge is $\overline{v_jv_k}$,
the subgraph $G'$ must contain both $v_j$ and $v_k$;
and it cannot contain $v_i$ since the only edge removed
is $\overline{v_iv_k}$. Recall that $v_j$ was the leftmost
vertex that $v_i$ is adjacent to; and by invariant~(I1),
we know $j=i+1$. Therefore, $G'=G_{jk'}$ for some $k'\geq k$,
and we have $|E_{jk'}| \geq 3|V_{jk'}| - 5$ after the shift.
Since $v_k$ was the rightmost vertex adjacent to $v_i$ before
the shift, all outgoing edges of $v_i$ went to vertices in $V_{jk'}$.
The outdegree of $v_i$ was at least $4$ before the shift,
hence $G_{ik'}$ had at least $3(|V_{ik'}|-1)-5+4 = 3|V_{ik'}| - 4 > 3|V_{ik'}| - 6$
edges, which is a contradiction.
\end {proof}

Now, after executing the algorithm, we are left with a multigraph
where the outdegree of every vertex is at most 3, and no subgraph
induced by $|V_{i,j}|\geq 3$ consecutive vertices has more than
$3|V_{ij}|-6$ edges. This, combined with invariant~(I1), implies that
the multiplicity of any edge $\overline{v_iv_{i+2}}$ is at most
one. Thus, for every vertex $v_i$, the (at most) three outgoing edges go to
vertices at distance at least $1$, $2$, and $3$, respectively, from $v_i$.
Denoting by $T(i)$ the number of $x$-monotone paths that start at $v_{n-i+1}$,
we arrive at the recurrence
\begin{equation*}
T(i) \le T(i-1) + T(i-2) + T(i-3), \text{ ~~for } i \geq 4,
\end{equation*}
with initial values $T(1)=T(2)=1$ and $T(3)=2$.
The recurrence solves to $T(n)=O(\alpha^n)$ where $\alpha =1.8392\ldots$
is the unique real root of the cubic equation $x^3-x^2-x-1=0$. Therefore,
every plane monotone graph on $n$ vertices admits $O(\alpha^n)$ $x$-monotone paths.
In particular, every plane straight-line graph on $n$ vertices admits
$O(\alpha^n)$ $x$-monotone paths.

Since the edges of an $n$-vertex planar straight-line graph have at
most $3n-6=O(n)$ distinct directions, the number of monotone paths
(over all directions) is bounded from above by $O(n\alpha^n)$.
We summarize our results as follows.

\begin{theorem}\label{thm:mon}
For every $n\in \mathbb{N}$, every triangulation on $n$ points contains
$O(\alpha^n)$ $x$-monotone paths and $O(n\alpha^n)$ monotone paths, where
$\alpha =1.8392\ldots$ is the unique real root of $x^3-x^2 -x -1=0$.
\end{theorem}

\subsection{Star-shaped polygons}

Given a plane straight-line graph $G$ on $n$ vertices,
the lines passing through the $O(n)$ edges of $G$
induce a line arrangement with $O(n^2)$ faces. Choose a face $f$ of the
arrangement, and a vertex $v$ of $G$. We show that $G$ contains
$O(\alpha^n)$ star-shaped polygons incident to vertex $v$ and with a
star center lying in $f$. Indeed, pick an arbitrary point $o\in f$.
Each edge of $G$ is oriented either clockwise or counterclockwise with
respect to $o$ (with the same orientation for any $o \in f$). Order the vertices
of $G$ by a rotational sweep around $o$ starting from the ray
$\overrightarrow{ov}$. Let $G_{f,v}$ be the graph obtained from $G$ by
deleting all edges that cross the ray $\overrightarrow{ov}$.
We can repeat our previous argument for monotone paths for $G_{f,v}$,
replacing the $x$-monotone order by the rotational sweep order about $o$,
and conclude that $G$ admits $O(\alpha^n)$ star-shaped polygons
incident to vertex $v$ and with the star center in $f$.
Summing over the $O(n)$ choices for $v$ and the $O(n^2)$ choices for $f$,
we deduce that $G$ admits $O(n^3 \alpha^n)$ star-shaped polygons.

\subsection{Directed simple paths}

Let $G=(V,E)$ be a directed planar
graph. Denote by $\degree^+(v)$ the outdegree of vertex $v\in V$;
let $V^+=\{v_1,\ldots, v_\ell\}$ be the set of vertices with outdegree
at least 1, where $1\leq \ell\leq n$. We show that for every $v \in V^+$,
there are $O(3^n)$ maximal (with respect to containment)
directed simple paths in $G$ starting from $v$.
Each maximal directed simple path can be encoded in an
$\ell$-dimensional vector that contains the outgoing edge of each
vertex $v\in V^+$ in the path (and an arbitrary outgoing edge if $v\in
V^+$ is not part of the path). The number of such vectors is
\begin{equation*}
\prod_{i=1}^{\ell} \degree^+(v_i) \leq \left(\frac{1}{\ell}
\sum_{i=1}^{\ell} \degree^+(v_i)\right)^\ell
< \left( \frac{3n}{\ell} \right)^\ell \leq 3^n,
\end{equation*}
where we have used the geometric-arithmetic mean inequality, and the fact that by Euler's
formula $\sum_{i=1}^\ell \degree^+(v_i)\leq 3n-6<3n$. We then have maximized
the function $x\rightarrow (3n/x)^x$ over the interval $1\leq x\leq n$.
Since there are $O(n)$ choices for the starting vertex $v \in V^+$,
and a maximal simple path contains $O(n)$ nonmaximal paths starting
from the same vertex, the total number of simple paths is $O(n^2 \, 3^n)$.

\section{Bounds on the minimum number of subgraphs}
\label{sec:min}

In this section, we explore the \emph{minimum} number of geometric
subgraphs of a certain kind that a triangulation on $n$ points in the plane can have.
We start with some easier results concerning convex polygons, star-shaped polygons,
and directed paths (in the next three subsections). The most difficult result
(concerning monotone paths) is deferred to the last subsection.

\subsection{Convex polygons}

Every $n$-vertex triangulation has at least $n-2$ triangular faces,
hence $n-2$ is a trivial lower bound for the number of convex polygons.
Hurtado, Noy and Urrutia~\cite{HNU99} proved that every triangulation contains
at least $\lceil n/2\rceil $ pairs of triangles whose union is convex, and
this bound is the best possible. Consequently, every triangulation contains
at least $3n/2 -O(1)$ convex polygons, each bounding one or two faces.
The $n$-vertex triangulations in Figure~\ref{fig:construction-min}\,(left)
contains $4n- O(1)$ convex polygons, and so this bound is the best possible apart
from constant factors. The triangulation consists of the join of two paths, $P_2 *P_{n-2}$,
where the path $P_{n-2}$ is realized as a monotone zig-zag path. Every convex
polygon is either a triangle or the union of two adjacent triangles
that share a flippable edge~\cite{HNU99}.
\begin{figure}[htpb]
\centering
\includegraphics[scale=0.75]{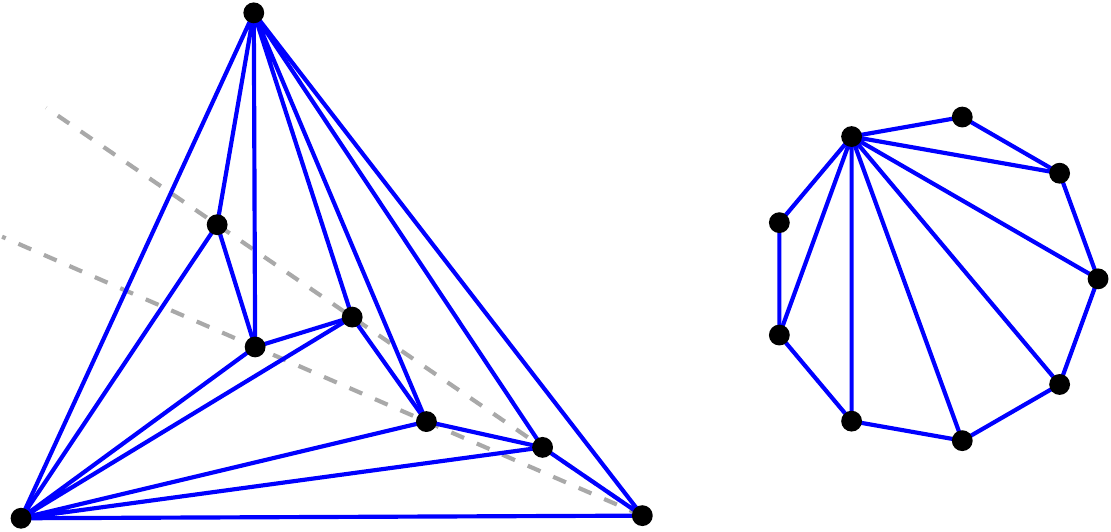}
\caption{There are $\Theta(n)$ convex polygons and $x$-monotone
  paths in the triangulation on the left; it contains $\Theta(n^4)$
  star-shaped polygons and monotone paths. There are $\Theta(n^2)$
  star-shaped polygons and $\Theta(n^4)$ monotone paths in the
  triangulation on the right.}
\label{fig:construction-min}
\end{figure}

\subsection{Star-shaped polygons}

Every convex polygon is star-shaped, and so the $3n/2 -O(1)$ lower bound
on the number of convex polygons in a triangulation on $n$ points
(from the previous paragraph) also holds for star-shaped polygons.
The following averaging argument yields an improved lower bound.

Consider a triangulation $T=(V,E)$ on $n$ points, $k$ of which are in
the interior of the convex hull of $V$. Then $T$ has $n+k-2$ bounded
(triangular) faces, $2n+k-3$ edges, $n+2k-3$ of which are interior edges.
Consequently, the average vertex degree in $T$ is $4+2(k-3)/n$.
For each vertex $v\in V$, every sequence of consecutive faces incident to $v$
forms a star-shaped polygon. The number of such sequences is
$2{\deg(v)\choose 2}+1$ for interior vertices and
${\deg(v)\choose 2}$ for boundary vertices. However,
summation over all $v\in V$ counts every triangular face three times,
and every quadrilateral formed by a pair of adjacent faces
twice. Consequently, the number of star-shaped polygons formed
by the union of faces with a common vertex is at least
\begin{align}
\sum_{v\in V} {\deg(v) \choose 2} & +k - 2(n+k-2) - (n+2k-3) \nonumber\\\
&\geq n {{4+2(k-3)/n}\choose 2} -3n-3k+7\nonumber\\
&\geq \frac{n}{2}\cdot \left(4+\frac{2(k-3)}{n}\right)\left(3+\frac{2(k-3)}{n}\right)-3n-3k+7\nonumber\\
&\geq 6n+7(k-3)+\frac{2(k-3)^2}{n} -3n-3k+7\nonumber\\
&= 3n+4k-14+\frac{2(k-3)^2}{n} \label{eq:degree}
\end{align}
where we have used Jensen's inequality in the first step.
Since $k \geq 0$, inequality~\eqref{eq:degree} yields at least $3n-O(1)$ star-shaped polygons.

Our best lower bound construction is a \emph{fan triangulation}, with a vertex of degree $n-1$,
shown in Figure~\ref{fig:construction-min}\,(right); it admits ${{n-1}\choose 2}$
star-shaped polygons.

\subsection{Directed paths}

In a directed triangulation, every edge is a directed path of length 1, and the boundary of
every triangular face contains at least one path of length two (that can be uniquely assigned to it).
Since every triangulation on $n$ points has at least $2n-3$ edges and at least $n-2$ triangular faces,
it follows that every directed triangulation has at least $3n-5$ directed paths of length 1 or 2.
This bound is the best possible.
Indeed, the directed triangulation shown in Figure~\ref{fig:construction-directed-min}
has $3n -5$ directed paths: the $n-1$ fan edges are directed upward and
the remaining edges on the convex hull are directed clockwise or counterclockwise,
in an alternating fashion. There are $2n-3$ paths of length $1$, and $n-2$ paths
of length $2$, each lying on the boundary of a triangular face, and there are no paths
of length $3$ or higher.
\begin{figure}[htpb]
\centering
\includegraphics[scale=1.05]{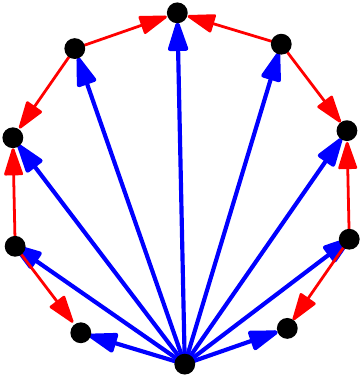}
\caption{There are $\Theta(n)$ directed paths in this directed planar triangulation.}
\label{fig:construction-directed-min}
\end{figure}

\subsection{Monotone paths}

\subsubsection{Lower bound}
We first argue that every triangulation contains $\Omega(n^2)$ monotone
paths, since there is a monotone path connecting any pair of vertices.
Indeed, this is a corollary of the following lemma applied to triangulations:
\begin{lemma} \textup{\cite{DRT12}[Lemma~1]}
Let $v$ be a vertex in a plane graph $G=(V,E)$ where every bounded face with $k\geq 3$
vertices is a convex $k$-gon, and the outer face is the exterior of the convex hull of $V$.
Then $G$ contains a spanning tree rooted at $v$ such that all paths starting at $v$ are monotone.
\end{lemma}
It is also worth noting that the monotone path connecting a pair of
vertices, $u$ and $v$, is not necessarily monotone in the direction
$\overrightarrow{uv}$. Also, two vertices are not always
connected by an $x$-monotone path: a trivial lower bound for the number
$x$-monotone paths is $\Omega(n)$, since every nonvertical edge is $x$-monotone.
The triangulation $P_2 *P_{n-2}$ in Figure~\ref{fig:construction-min}\,(left)
is embedded such that the path $P_{n-2}$ is $x$-monotone and lies to
the right of $P_2$. With this embedding, it contains $\Theta(n^2)$
$x$-monotone paths: every $x$-monotone path consists of a sequence of
consecutive vertices of $P_{n-2}$, and 0, 1, or 2 vertices of $P_2$. However,
both triangulations in Figure~\ref{fig:construction-min} admit $\Theta(n^4)$
monotone paths (over all directions).

Triangulations with a polynomial number of monotone paths are also
provided by known constructions in which all monotone paths are
``short''. Dumitrescu, Rote, and T\'oth~\cite{DRT12} constructed
triangulations with maximum degree $O(\log n/\log \log n)$ such that
every monotone path has $O(\log n/\log n\log n)$ edges. Moreover,
there exist triangulations with bounded vertex-degree in which every monotone
path has $O(\log n)$ edges. These constructions contain polynomially
many, but $\omega(n^4)$, monotone paths.

\subsubsection{Upper bound}
We construct a triangulation $T$ of $n$ points containing
$O(n^{2\log 3} \log^2 n) = O(n^{3.17})$ monotone
paths\footnote{Throughout this paper, all logarithms are in base $2$.}.
A~similar construction was introduced in~\cite{DRT12} as a stacked
polytope in $\mathbb{R}^3$, where every monotone path on its 1-skeleton
has $O(\log n)$ edges.
\begin{theorem}\label{thm:monpath}
For every $n\in \mathbb{N}$, there is an $n$-vertex triangulation that contains
$O(n^{2\log 3}\log^2 n)=O(n^{3.17})$ monotone paths.
\end{theorem}

\paragraph{Construction.}
For every integer $\ell\geq 0$, we define a triangulation $T$ on $n=2^\ell+2$ vertices.
Refer to Figure~\ref{fig:mp}. The outer face is a right triangle $\Delta{oab}$,
where $o$ is the origin and $\angle boa =\frac{\pi}{2}$. The interior vertices are
arranged on $\ell$ circles, $C_0,C_1,\ldots, C_{\ell-1}$, centered at the
origin, with the radii of the circles rapidly approaching $0$.
We place $2^i$ points on $C_i$, in an equiangular fashion, as described below,
and so the number of interior points is $\sum_{i=0}^{\ell-1} 2^i= 2^\ell-1$.
\begin{figure}[htbp]
\centering
\includegraphics[scale=0.9]{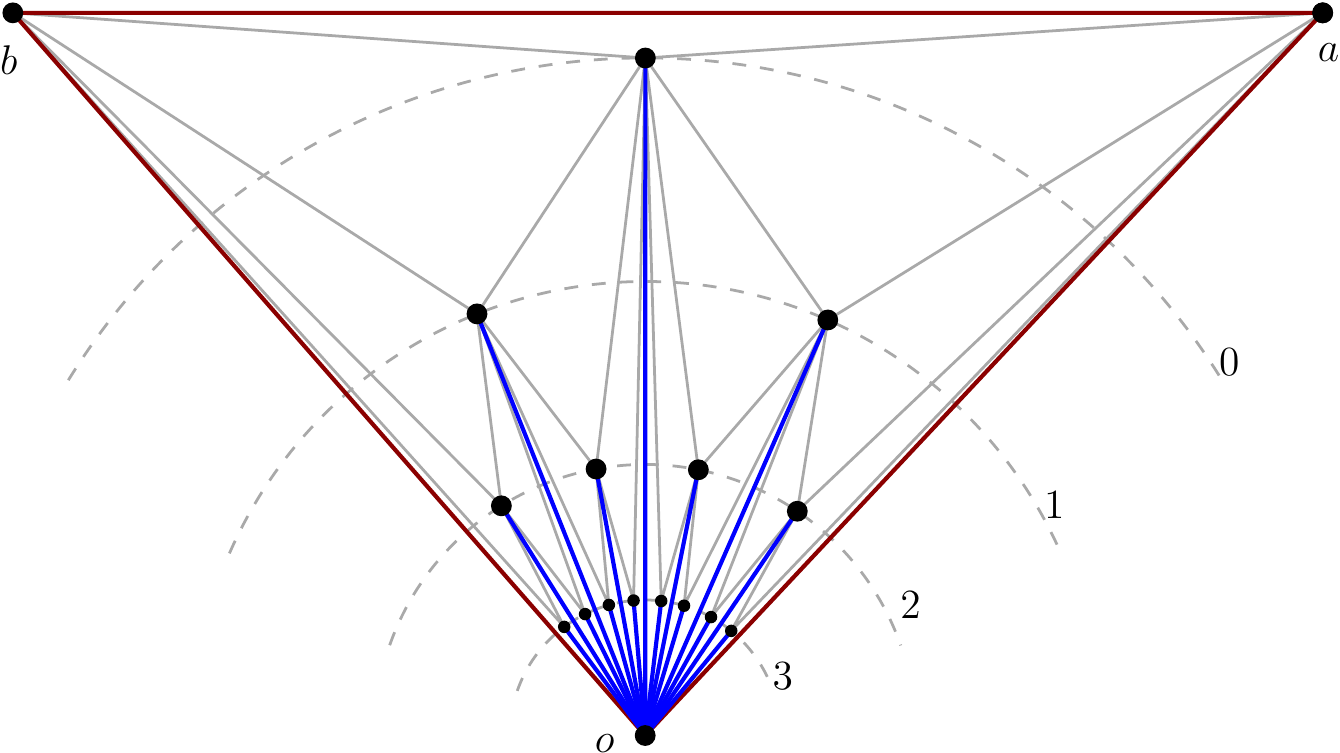}
\caption{A schematic illustration of the triangulation $T$. The radii
  of the circles $C_i$, $i=0,1,\ldots,\ell-1$, converge to 0 much faster
  than indicated in the figure.}
\label{fig:mp}
\end{figure}

The rays to the points on $C_i$ are interspersed with the rays to
points on all previous layers. Specifically, the $2^i$ points on
circle $C_i$ are incident to rays emitted from the origin in
directions $\frac{\pi}{4}+\frac{2j-1}{4\cdot 2^{i}} \pi$ for $j=1,\ldots , 2^i$.
The edges of the triangulation are defined as follows.
The origin $o$ is connected to all other vertices. Each vertex on circle $i$ is
connected to the two vertices of the previous layers that are closest
in angular order. The radii of the circles $C_i$ are chosen
recursively for $i=0,1,\ldots,\ell-1$, such that the edges that
connect a vertex $v\in C_i$ to vertices $v'$ and $v''$ on larger circles
are almost parallel to $ov'$ and $ov''$, respectively. Specifically, we require
that $\angle vv'o<\pi/2^{\ell+1}$ and $\angle vv''o<\pi/2^{\ell+1}$
(so these angles are less than the angle between any two
consecutive edges incident to $o$).

\paragraph{Maximal monotone paths.}
In the argument, we sometimes focus on monotone paths that are \emph{maximal}
(with respect to containment). This is justified by the following easy lemma.
\begin{lemma}\label{lem:maximalmon}
Let $T$ be a triangulation of a point set $S$ in the plane.
\begin{itemize} \itemsep -1pt
\item[{\rm (i)}] The two endpoints of a maximal monotone path in $T$ are vertices
  of the convex hull $\conv(S)$.
\item[{\rm (ii)}] If $T$ contains $m$ maximal monotone paths, and every
  such path has at most $k$ vertices, then the total number of
  monotone paths in $T$ is at most $m {k\choose 2}$.
\end{itemize}
\end{lemma}
\begin{proof}
(i) Let $\xi=(v_1,\ldots , v_t)$ be a maximal path in $T$ that is monotone
in direction $\mathbf{u}\in \mathbb{R}^2 \setminus \{\mathbf{0}\}$.
Suppose that the endpoint $v_t$ lies in the interior of $\conv(S)$.
Note that the angle between any two consecutive edges of $T$ incident to
an interior point of $\conv(S)$, to $v_t$ in particular, is less than $\pi$, since
each angle is an interior angle of a triangle.
Consequently, there is an edge $v_tw$ in $T$ such that
$\langle \overrightarrow{v_tw},\mathbf{u}\rangle>0$. Now
the path $(v_1,\ldots, v_t,w)$ is monotone in direction $\mathbf{u}$,
and strictly contains $\xi$, in contradiction with the maximality of $\xi$.

\smallskip
\noindent (ii) Every subpath of a monotone path is also monotone (in
the same direction). A path with $t \geq 2$ vertices has exactly ${t \choose 2}$
subpaths, each determined by the two endpoints, and so the claim follows.
\end{proof}

\paragraph{Analysis.}
We say that a (directed) edge or a path is \emph{upward} if it is $(0,1)$-monotone,
and \emph{downward} if it is $(0,-1)$-monotone. (The horizontal edge $ab$ is
neither upward nor downward). Clearly, every upward path is monotone
in direction $(0,1)$, but a path that contains both upward and
downward edges may be monotone in some other direction $\mathbf{u}\in
\mathbb{R}^2\setminus\{\mathbf{0}\}$. Since $o$ is the point with the
minimum $y$-coordinate in $T$, every maximal upward path starts from
$o$ and ends at $a$ or $b$. Consequently every upward path has at most
one vertex on each circle, thus at most $\ell+2=O(\log n)$ vertices
overall. In the remainder of the analysis, we bound the number of
upward paths in $T$, and then show that every monotone path is composed of
a small number (bounded by a constant) of upward and downward subpaths.

\begin{lemma} \label{lem:upwardpath}
The number of maximal upward paths in $T$ is $O(n^{\log 3})$.
The number of upward paths is $O(n^{\log 3} \log n) =O(n^{1.585})$.
\end{lemma}
\begin{proof}
Let $\tau_i$ denote the total number of upward paths that start from
a point on the circle $C_i$ and end at $a$ or $b$. Observe that $\tau_0=2$
and $\tau_1=3+3=6$. Recall that by construction, each vertex
on $C_i$ is connected to the two vertices of the previous layers that are closest
in angular order (left and right). It follows that
\begin{equation} \label{eq:21}
\tau_i= 2 \sum_{j=0}^{i-1} \tau_j  + 2, \text{ for } i \geq 1,
\end{equation}
where the term $2$ counts the direct edges to $a$ and $b$ from the
leftmost and the rightmost points of $C_i$, respectively.

We now prove that $\tau_i= 2 \cdot 3^i$ for $i=0,1,\ldots,\ell-1$.
We proceed by induction on $i$. The base case $i=0$ is satisfied as
verified above by the value $\tau_0=2$. For the induction step,
assume that the formula holds up to $i$. According to~\eqref{eq:21} we have
$$\tau_{i+1}
= 2 \sum_{j=0}^{i} \tau_j  + 2
= 2 \left( 2 \sum_{j=0}^i 3^j \right) +2
= 2 \cdot 2 \cdot \frac{3^{i+1} -1}{2} + 2
= 2 \cdot 3^{i+1} -2 + 2 = 2 \cdot 3^{i+1},
$$
as required. Write $\tau= \sum_{i=0}^{\ell-1} \tau_i$. Recall that $n=2^\ell+2$, hence
\begin{equation} \label{eq:22}
\tau = \sum_{i=0}^{\ell-1} \tau_i \leq 3^\ell \leq 3^{\log{n}} = n^{\log{3}}.
\end{equation}

It follows that the number of upward paths starting from
vertices on $\bigcup_{i=0}^{\ell-1} C_i$ and ending at $a$ or $b$
is $\tau$, and further, that the number of upward paths
from $o$ or $\bigcup_{i=0}^{\ell-1} C_i$ to $a$ or $b$ is  $2\tau$.
Since every such path has at most $\ell+2$ vertices, the total number of
upward paths is at most $(\ell +1)\tau=O(n^{\log 3}\log n)$.
\end{proof}

\begin{corollary} \label{cor:from-o}
The number of upward paths in $T$ starting from $o$ is
$O(n^{\log 3}) = O(n^{1.585})$.
\end{corollary}
\begin{proof}
Let $v_0$ be a vertex on $C_i$, incident to two upward edges $(v_0,v_1)$ and $(v_0,v_2)$.
Then all upward paths from $o$ to $v_0$ lie in the union of two triangles
$\Delta{o v_0 v_1} \cup \Delta{o v_0 v_2}$.
The subgraph of $T$ lying in (the interior or on the boundary of) each
of these triangles is isomorphic to the analogue of $T$ on
$2^{\ell-i}+2$ vertices, with the same upward-downward
orientations. As shown in the proof of Lemma~\ref{lem:upwardpath},
there are $O(3^{\ell-i})$ upward paths from $o$ to $v_0$ in
each of the triangles $\Delta{o v_0 v_1}$ and $\Delta{o v_0 v_2}$,
and so $T$ contains $O(3^{\ell-i})$ upward paths from $o$ to $v_0$.
Summing over all vertices of $T$, the total number of upward paths starting from $o$ is
$$ O\left( \sum_{i=0}^{\ell-1} 2^i \cdot 3^{\ell-i} \right) =
O\left( 3^\ell \sum_{i=0}^{\ell-1} \left( \frac23 \right)^i \right) =
O\left(3^\ell \right) = O\left(3^{\log n}\right) = O\left(n^{\log 3}\right), $$
as claimed.
\end{proof}

Lemma~\ref{lem:upwardpath} also yields the following.
\begin{corollary} \label{cor:upwardpath}
The number of monotone paths in $T$ composed of one upward path and one downward path
is $O(n^{2\log 3} \log^2 n) = O(n^{3.17})$.
\end{corollary}

It remains to consider monotone paths in $T$ that cannot be decomposed into
one upward path and one downward path.\footnote{Such paths were
  inadvertently overlooked in the analysis from~\cite{LST13}.}
Every such path changes vertical direction (upward vs downward) at least twice.
\begin{figure}[htbp]
\centering
\includegraphics[scale=0.87]{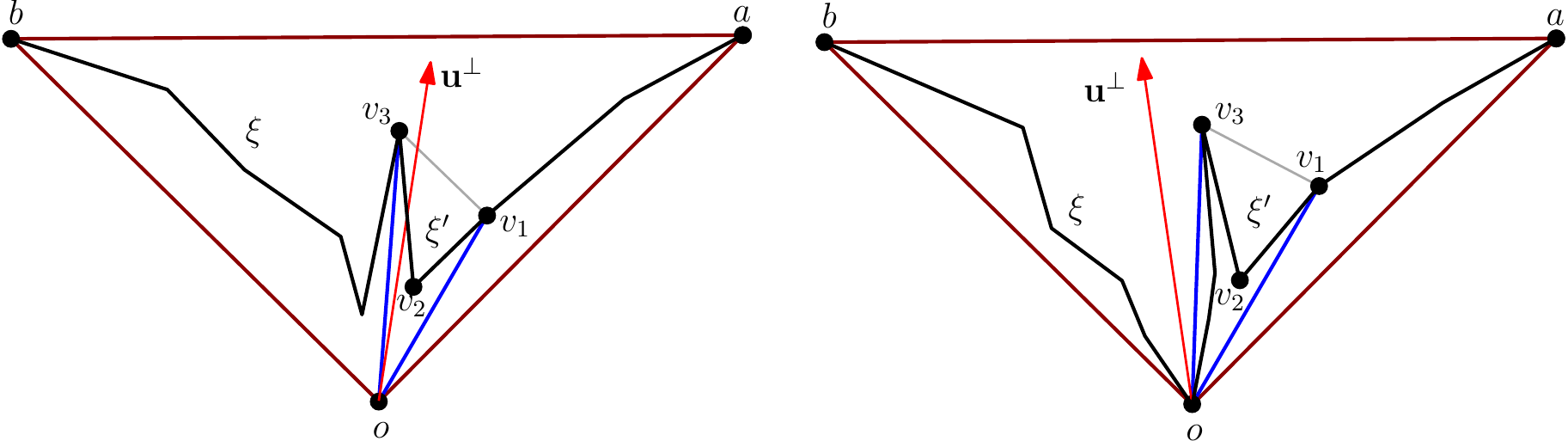}
\caption{Two schematic pictures of a subpath $\xi'=(v_1,v_2,v_3)$ in $T$
  such that the edge $(v_1,v_2)$ is downward, the edge $(v_2,v_3)$ is upward, and $v_2\neq o$.
Left: A maximal monotone path $\xi$ containing $\xi'$, where $\xi \setminus \xi'$
lies in the exterior of $\Delta{o v_1 v_3}$.
Right: A maximal monotone path $\xi$ containing $\xi'$, where part of
$\xi \setminus \xi'$ lies inside $\Delta{o v_1 v_3}$.}
\label{fig:v123}
\end{figure}

\begin{lemma}\label{lem:modus}
The total number of monotone paths in $T$ that change vertical
direction at least twice is $O(n^{2\log 3} \log^2 n) = O(n^{3.17})$.
\end{lemma}
\begin{proof}
We first characterize the monotone paths that change vertical direction at least twice,
showing that in fact they change vertical direction at most three times;
we then derive an upper bound on their number.

Let $\xi$ be a maximal monotone path in $T$ that changes vertical direction at least twice.
Refer to Figure~\ref{fig:v123}. By Lemma~\ref{lem:maximalmon}(i), the
endpoints of $\xi$ are vertices of the outer face $\Delta{oab}$. One
of the endpoints of $\xi$ is $a$ or $b$, where all upward edges point
to $a$ or $b$, respectively. Consequently, $\xi$ changes vertical
directions at a vertex in the interior of $\Delta{oab}$ from downward
to upward. That is, $\xi$ contains a subpath $\xi'=(v_1,v_2,v_3)$ such
that $(v_1,v_2)$ is downward, $(v_2,v_3)$ is upward, and $v_2\neq
o$. By symmetry, we can assume that $v_1$ and $v_3$ are on the right
and left sides of the ray $\overrightarrow{ov_2}$, respectively.
By construction, $v_1v_3$, $ov_1$, and $ov_3$ are edges of $T$, so $\Delta{o v_1 v_3}$ is a
3-cycle in $T$.

Since both $\overrightarrow{v_1v_2}$ and $\overrightarrow{v_2v_3}$ are
$\mathbf{u}$-monotone, for some vector $\mathbf{u}\neq \mathbf{0}$, the
orthogonal vector $\mathbf{u}^\perp$ lies between the directions of
$\overrightarrow{v_2v_1}$ and $\overrightarrow{v_2v_3}$.
By construction, these directions are within $\pi/2^{\ell+1}$ from the
directions of $\overrightarrow{ov_1}$ and $\overrightarrow{ov_3}$,
respectively. For every such vector $\mathbf{u}$, vertex $a$ is $\mathbf{u}$-minimal
and $b$ is $\mathbf{u}$-maximal, and so $\xi$ is a path from $a$ to $b$.

The edges of $\xi$ directly preceding and following
$\xi'=(v_1,v_2,v_3)$ have a crucial role in determining the edges in $\xi \setminus \xi'$.
The proof of Lemma~\ref{lem:modus} relies on the following two claims.

\begin{claim}\label{cl:1}
If $\xi$ does not change vertical direction at $v_1$, then $v_1$
uniquely determines the part of $\xi$ preceding $v_1$.
Similarly, if $\xi$ does not change vertical direction at $v_3$, then
$v_3$ uniquely determines the part of $\xi$ following $v_3$.
\end{claim}
\begin{proof}
By symmetry, it is enough to prove the second statement. If $v_3=b$,
the proof is complete.
Assume that $v_3\neq b$, and label the portion of $\xi$ starting at
$v_3$ by $(v_3,v_4,\ldots ,v_t)$.  Recall that by construction, from every interior vertex $v$,
there are two upward edges lying on opposite sides of the ray $\overrightarrow{ov}$.

We show by induction on $j=3,\ldots ,t$ that the edge $(v_j,v_{j+1})$ is upward, it lies on
the left of $\overrightarrow{ov_j}$, and the other upward edge starting from $v_j$
is incident to a vertex on or to the right of $\overrightarrow{ov_1}$.
In the base case, $j=3$, and by assumption $(v_3,v_4)$ is upward.
The upward edge on the right side of $\overrightarrow{ov_3}$ cannot enter the interior
of $\Delta{o v_1 v_3}$ and so it must go to a vertex on or to the right of $\overrightarrow{ov_1}$.
Thus this edge is not $\mathbf{u}$-monotone, and so $(v_3,v_4)$ must be the upward edge
that leaves $v_3$ on the left side of $\overrightarrow{ov_3}$.

Assume now $3<j<t$ and the induction hypothesis holds for $v_{j-1}$.
By construction, the downward edges staring from $v_j$ are almost parallel
to $\overrightarrow{v_jo}$, and so they are not $\mathbf{u}$-monotone.
By induction, an upward edge $(v_{j-1},r)$ is incident to a vertex on or to the right of
$\overrightarrow{ov_1}$. Since the two upward edges $(v_{j-1},v_j)$ and $(v_{j-1},r)$ are
incident to the same triangle of $T$, it follows that $(v_j,r)$ is an edge in $T$.
Hence, the upward edge on the right side of $\overrightarrow{ov_j}$ also
goes to a vertex on or to the right of $\overrightarrow{ov_1}$, and so it is not
$\mathbf{u}$-monotone. Consequently, the $\mathbf{u}$-monotone path $\xi$ leaves $v_j$
on the unique upward edge on the left side of $\overrightarrow{ov_j}$.
\end{proof}

\begin{claim}\label{cl:2}
If $\xi$ arrives at $v_1$ on an upward edge, then $v_3=b$ or $\xi$
leaves $v_3$ on an upward edge.
Similarly, if $\xi$ leaves $v_3$ on a downward edge, then $v_1=a$ or
$\xi$ arrives at $v_1$ on a downward edge.
\end{claim}
\begin{proof}
By symmetry, it is enough to prove the second statement. Assume that
$\xi$ leaves $v_3$ on a downward edge $(v_3,v_4)$. Then the directions of
$(v_3,v_4)$ and $(v_3,v_2)$ are both within $\pi/2^{\ell+1}$ from
$\overrightarrow{v_3o}$. Thus $\mathbf{u}^\perp$ is within $\pi/2^{\ell+1}$ from
$\overrightarrow{ov_3}$. This, in turn, implies that none of the upward edges
going into $v_1$ is $\mathbf{u}$-monotone, and the claim follows.
\end{proof}

We continue with the proof of Lemma~\ref{lem:modus}. If the path $\xi$
changes vertical direction at neither $v_1$ nor $v_3$, then
Claim~\ref{cl:1} implies that $\xi$ is composed of one upward path and
one downward path, contradicting our assumption that $\xi$ has no such
decomposition.

Assume that the path $\xi$ changes vertical direction at $v_1$ or $v_3$.
Without loss of generality, assume that $\xi$ changes vertical
direction at $v_3$, \ie, $\xi$ leaves $v_3$ on a downward edge.
By Claims~\ref{cl:1} and~\ref{cl:2}, the part of $\xi$ preceding $v_1$ is uniquely determined.
Note also that $\mathbf{u}^\perp$ is within $\pi/2^{\ell+1}$ from $\overrightarrow{ov_3}$.

We distinguish two cases based on whether the edge $(v_3,v_4)$ of $\xi$ following $v_3$
lies in the exterior of the triangle $\Delta{o v_1 v_3}$ or in its interior;
then we estimate the number of such maximal paths $\xi$ and their subpaths.

\smallskip\noindent{\bf Case~1: $v_4$ lies in the exterior of $\Delta{o v_1 v_3}$;}
refer to Figure~\ref{fig:v123}\,(left).
Then the vertices of $\xi$ following $v_4$ are uniquely determined,
analogously to Claim~\ref{cl:1}. We have $O(n)$ choices for $v_2$
(which determines the triple $(v_1,v_2,v_3)$), and
$O(\log n)$ choices for $v_4$. Consequently, $T$ contains
$O(n \log n)$ maximal monotone paths $\xi$ of this type.
The length of any such path is $O(\log n)$, and by Lemma~\ref{lem:maximalmon}(ii),
the total number of monotone paths of this type is  $O(n \log^3 n)$.

\smallskip\noindent{\bf Case~2: $v_4=o$ or $v_4$ lies in the interior of $\Delta{o v_1 v_3}$;}
refer to Figure~\ref{fig:v123}\,(right).
Then the path $\xi$ reaches the $\mathbf{u}$-maximal vertex of
$\Delta{o v_1 v_3}$, namely $o$. The portion of $\xi$ inside $\Delta{o v_1 v_3}$
is a downward path from $v_3$ to $o$. The portion of $\xi$ from
$o$ to $b$ must be an upward path. In summary, the maximal monotone path $\xi$
is composed of four upward or downward paths: A downward path from $a$ to $v_2$,
an upward edge $(v_2,v_3)$, a downward path from $v_3$ to $o$, and an
upward path from $o$ to $b$.

Let us count the number of maximal monotone paths $\xi$ of this type.
By Lemma~\ref{lem:upwardpath}, there are $O(n^{\log 3})$ choices
for the portion of $\xi$ from $o$ to $b$, and $O(n^{\log 3})$ independent
choices for the portion from $v_3$ to $o$. Since the degree of $v_3$
is $O(\log n)$, we have $O(\log n)$ choices for vertex $v_2$. Finally,
the portion from $a$ to $v_2$ is uniquely determined by $v_2$ by Claim~\ref{cl:1}.
This gives an $O(n^{2\log 3}\log n)$ bound on the number of maximal monotone
paths of this type. The length of any such path is $O(\log n)$.

By Lemma~\ref{lem:maximalmon}(ii), the total number of monotone paths of this type is
$O(n^{2\log 3} \log^3 n)$. By counting the subpaths of the maximal
monotone paths $\xi$ directly, we can reduce this bound by a logarithmic factor.
If a subpath of $\xi$ changes vertical directions twice, then its two
endpoints must lie in the first and last portion of $\xi$, respectively.
By Corollary~\ref{cor:from-o}, there are $O(n^{\log 3})$ choices for the portion of $\xi$
from $o$ to $b$ \emph{and} for all subpaths of $\xi$ incident to $o$.
The length of this last portion is $O(\log n)$, and so the number of its subpaths incident
to $v_2$ is $O(\log n)$. Consequently, the total number of monotone paths of this type
is $O(n^{2\log 3}\log^2 n)$.

\smallskip Summing over both cases, it follows that the number of monotone
paths that change vertical directions at least twice is $O(n^{2\log 3}\log^2 n)$, as claimed.
\end{proof}

\paragraph{Proof of Theorem~\ref{thm:monpath}.}
The triangulation $T$ contains $O(n^{\log 3}\log n)$ upward paths by
Lemma~\ref{lem:upwardpath}, and $O(n^{2\log 3}\log^2 n)$ paths composed
of one upward and one downward piece by Corollary~\ref{cor:upwardpath}.
Any other monotone path $\xi=(v_1,\ldots ,v_t)$ changes vertical
direction (upward vs downward) at least twice. By Lemma~\ref{lem:modus},
$T$ contains $O(n^{2\log 3}\log^2 n)$ such paths. Consequently, the total
number of monotone paths in $T$ is $O(n^{2\log 3}\log^2 n) = O(n^{3.17})$, as claimed.
\qed

\section{Conclusion}

We have derived estimates on the maximum and minimum number of star-shaped polygons,
monotone paths, and directed paths that a (possibly directed) triangulation of
$n$ points in the plane can have. Our results are summarized in
Tables~\ref{tab:results} and~\ref{tab:results-min} of Section~\ref{sec:intro}.
Closing or narrowing the gaps between the upper and lower bounds remain as interesting
open problems. The gaps in the last row of Table~\ref{tab:results} and
2nd row of Table~\ref{tab:results-min} are particularly intriguing.

\section*{Acknowledgments}

A. Dumitrescu was supported in part by NSF grant DMS-1001667.
M. L\"offler was supported by the Netherlands Organization for
Scientific Research (NWO) under grant 639.021.123. Research by T\'oth
was supported in part by NSERC (RGPIN 35586) and NSF (CCF-1423615).
This work was initiated at the workshop ``Counting and Enumerating 
Plane Graphs,'' which took place at Schloss Dagstuhl in March, 2013.

\end {document}